\documentclass[a4paper, english, 10pt]{amsart}

\usepackage{amsfonts}
\usepackage{caption}
\usepackage{forest}
\usepackage{amsmath}
\usepackage{amssymb}
\usepackage{tikz}
\usepackage[cp850]{inputenc}
\usepackage[bookmarksnumbered,plainpages
]{hyperref}
\usepackage{color}
\usepackage{mathtools}
\usepackage{MnSymbol}
\usepackage{subfigure}
\usepackage[all]{xy}
\usepackage{amsaddr}

\makeatletter
\@namedef{subjclassname@2020}{%
  \textup{2020} Mathematics Subject Classification}
\makeatother

\allowdisplaybreaks
\newtheorem{theorem}{Theorem}[section]

\newtheorem{Proposition}[theorem]{Proposition}

\newtheorem{Lemma}[theorem]{Lemma}

\newtheorem{proposition}[theorem]{Proposition}
\newtheorem{corollary}[theorem]{Corollary}

\newtheorem{lemma}[theorem]{Lemma}
\newtheorem{conjecture}[theorem]{Conjecture}

\newtheorem{definition}[theorem]{Definition}
\newtheorem{example}[theorem]{Example}
\newtheorem{problem}{Problem}

\def\sym#1{\mathrm{Sym}(#1)}
\def\c#1{\mathrm{con}_{#1}}
\def\comment#1{{\color{red} #1}}
\def\aut#1{\mathrm{Aut}(#1)}
\def\End#1{\mathrm{End}(#1)}
\def\aff#1{\mathrm{Aff}#1}
\def\Aff#1{\mathrm{Aff}#1}
\def\lmlt{\mathrm{LMlt}}
\def\N{Norm}
\def\dis{\mathrm{Dis}}
\def\Q{\mathcal{Q}}

\def\setof#1#2{\{#1\, : \,#2\}}

\newcommand*\xbar[1]{%
   \hbox{%
     \vbox{%
       \hrule height 0.5pt 
       \kern0.5ex
       \hbox{%
         \kern-0.1em
         \ensuremath{#1}%
         \kern-0.1em
       }%
     }%
   }%
}

\def\m{\mathfrak{ip} }

\title{Superconnected left quasigroups and involutory quandles}

\author{M. Bonatto}

\address[M. Bonatto]{Dipartimento di matematica e informatica - UNIFE}

\email{marco.bonatto.87@gmail.com}
\begin{document}


	\begin{abstract}
	In this paper we study the classes of superconnected and superfaithful left quasigroups, that are relevant in the study of Mal'cev varieties of left quasigroups \cite{Maltsev_paper}. Then we focus on quandles and in particular to the involutory ones. We extend the main result of \cite{involutive_quandles_russo} to the infinite case and we offer a characterization of several classes of involutory quandles in terms of the properties of the canonical generators of the displacement group, improving the main results of \cite{Nobu}.
	\end{abstract}
	

\maketitle

\section*{Introduction}

 Algebraic structure of interest in many areas of mathematics often have an underlying {\it left quasigroup} structure. Examples are {\it quandles} that arise in low dimensional topology \cite{J, Matveev}  and the algebraic structure related to the solution of the {\it Yang-Baxter equation} \cite{Rump, Rumples}. The goal of this paper is to keep developing some tools for understanding left quasigroups as started in \cite{semimedial}. In this paper we study the class of {\it superfaithful} and the class of {\it superconnected} left quasigroups. Such notions arise naturally in the framework of {\it Mal'cev conditions} for left quasigroups that we study in a separate paper \cite{Maltsev_paper}.

In some sense superfaithful and superconnected left quasigroups are close to {\it quasigroups}. Indeed latin left quasigroups (i.e. left quasigroup reducts of quasigroups) are superfaithful and connected and the finite ones are also superconnected (the converse is not true). On the other hand, superconnected left quasigroups have a Mal'cev term \cite{Maltsev_paper}.

For quandles, the property of being {\it connected} is topologically relevant (as connected quandles provide knot invariants). The results of this paper and of \cite{Maltsev_paper} suggests that such property is relevant also from an algebraic viewpoint. Indeed, several results on finite latin quandles can be extended to the class of superconnected quandles. For instance, the commutator theory in the sense of \cite{comm} is particularly well-behaved in this class (see Proposition \ref{commutator for superconnected}). In some cases, superconnected quandles are indeed latin, as the nilpotent (see Theorem \ref{nilpotence}) and the {\it involutory} ones (Theorem \ref{involutory superfaith} improves the main result of \cite{involutive_quandles_russo} and partially the main result of \cite{Nobu} that were limited to the finite case). 

Involutory quandles encode the notion of {\it symmetric space} as defined in \cite{Loos} and they are also related to {\it Bruck loops} \cite{Stanos,Petr2}. In the last Section we show that some properties of involutory quandles are determined by the properties of the canonical generators of the {\it displacement group} partially inspired by \cite{Nobu} (see Theorem \ref{involutory superfaith} and Theorem \ref{reductive involutory 2}). As a byproduct we obtain some group theoretical applications on finite groups generated by a conjugacy class of involutions (see Corollaries \ref{group1} and \ref{group2}).

The paper is organized as follows: in Section \ref{Sec:left} we collect all the basic definitions needed in the sequel of the paper, and in \ref{Sec:connected} and \ref{Sec:Idempotent} we collect some basic results on connected and idempotent left quasigroups, respectively (including two characterization of superconnected left quasigroups in Lemma \ref{2 gen involutory} and Corollary \ref{cor_superconnected}). Section \ref{Sec:Racks} is dedicated to racks and quandles. In Section \ref{Sec:superfaithful} we show some construction of (infinite families of) superfaithful quandles and in Sections \ref{Sec:Superconnected} and \ref{Sec:commutator} we explore superconnected quandles. We conclude the paper with Section \ref{Sec:Involutory} about involutory quandles.

We used the software Prover9 \cite{Prover9} to compute some of the examples appearing in the paper and the \cite{RIG} library of GAP as a source of concrete examples.

\section*{Acknowledgments}
The author is grateful to Prof. Gast\'{o}n Garc\'{i}a for the fruitful discussion in La Plata around the notion of superconnected quandle and related problems. The author would also like to thank Prof. David Stanovsk\'{y} for pointing him out the relevant paper \cite{Nobu}.

\section{Preliminary results}

%

\subsection{Left quasigroups}\label{Sec:left}

A left quasigroup is a binary algebraic structure $(Q,\ast,\backslash)$ such that the identities
\begin{displaymath}
x\backslash (x\ast y)\approx y\approx x\ast (x\backslash y)
\end{displaymath}
hold, i.e. the left multiplications $L_a:b\mapsto a\ast b$ are bijective for every $a\in Q$. The dual notion of right quasigroup is defined analogously. The left multiplication group of $Q$ is $\lmlt(Q)=\langle\setof{L_a}{a\in Q}\rangle$.  \\
We denote by $\textbf{H}(Q)$, $\textbf{S}(Q)$ and $\textbf{P}(Q)$ respectively the set of isomorphism classes of homomorphic images, subalgebras and powers of the left quasigroup $Q$. Let $X$ be a subset of $Q$, we denote by $Sg(X)$ the smallest subalgebra of $Q$ containing $X$.

A {\it congruence} of a left quasigroup $Q$ is a equivalence relation $\alpha$ such that the implication
\begin{equation}\label{cong}
a\,\alpha\, b\text{ and } c\,\alpha\, d\, \Rightarrow\, (a*c)\,\alpha\, (b*d) \text{ and } (a\backslash c)\, \alpha\, (b\backslash d)
\end{equation}
holds for every $a,b,c,d\in Q$. Congruences and homomorphic images are essentially the same thing because of the second isomorphism theorem for arbitrary algebraic structures \cite{UA}. Indeed if $\alpha$ is a congruence, the operations
$$[a]_\alpha * [b]_\alpha=[a*b]_\alpha\, \quad [a]_\alpha \backslash  [b]_\alpha=[a\backslash b]_\alpha$$
for every $[a]_\alpha, [b]_\alpha\in Q/\alpha$ are well-defined by virtue of \eqref{cong} and  the quotient set $Q/\alpha$ is a left quasigroup with respect to such operations. On the other hand if $h:Q\mapsto Q'$ is a left quasigroup homomorphism, then $\ker{h}=\setof{(a,b)\in Q^2}{h(a)=h(b)}$ is a congruence of $Q$ and $Im(h)\cong Q/\ker{h}$. The congruences of $Q$ form a lattice denoted by $Con(Q)$ with minimum $0_Q=\setof{(a,a)}{a\in Q}$ and maximum $1_Q=Q\times Q$. If $\alpha$ is a congruence of $Q$, the congruence lattice of $Q/\alpha$ is given by $\setof{\beta/\alpha}{\alpha\leq\beta \in Con(Q)}$, where 
$$[a]_\alpha\, \beta/\alpha\, [b]_\alpha\, \text{ if and only if } \, a\,\beta\, b.$$
Moreover, the mapping
\begin{displaymath}
\pi_{\alpha}:\lmlt (Q)\longrightarrow \lmlt(Q/\alpha),\quad L_{a_1}^{k_1}\ldots L_{a_n}^{k_n} \mapsto L_{[a_1]}^{k_1}\ldots L_{[a_n]}^{k_n},
\end{displaymath}
is a well defined surjective homomorphism of groups (see \cite[Lemma 1.8]{AG} for racks and \cite{CP} for left quasigroups). Moreover,
\begin{equation}\label{formula for pi alpha}
[h(a)]_\alpha=\pi_\alpha(h)([a]_\alpha)
\end{equation}
holds for every $a\in Q$ and every $h\in \lmlt(Q)$.

The {\it displacement group relative to a congruence $\alpha$} is the smallest normal subgroup of $\lmlt(Q)$ containing $\setof{L_a L_b^{-1}}{a\,\alpha\,b}$ (see \cite[Section 3.1]{CP}) i.e. 
$$\dis_\alpha=\langle h L_a L_b^{-1} h^{-1}, \, a\,\alpha \, b, \, h\in \lmlt{(Q)}\rangle.$$
For $\alpha=1_Q$ we denote the relative displacement group as $\dis(Q)$ and we call it the {\it displacement group of $Q$}. 
\begin{lemma}\cite[Lemma 1.4]{semimedial}\label{disQ combinatorial} 
Let $Q$ be a left quasigroup. Then
\begin{equation*}
\dis(Q)=\setof{L_{x_1}^{k_1},\ldots L_{x_n}^{k_n}}{x_1,	\ldots, x_n\in Q,\,\sum_{i=1}^n k_i=0}
\end{equation*}
and in particular $\lmlt(Q)=\dis(Q)\langle L_a\rangle$ for every $a\in Q$.
\end{lemma}
If $\alpha, \beta$ are congruences of a left quasigroup $Q$ and $\alpha\leq \beta$, the image of $\dis_\beta$ under $\pi_\alpha$ is $\dis_{\beta/\alpha}$ and in particular the restriction of $\pi_\alpha$ to $\dis (Q)$ gives a surjective homomorphism $\dis (Q)\to\dis(Q/\alpha)$. The kernels of $\pi_\alpha$ and of its restriction will be denoted respectively by $\lmlt^\alpha$ and $\dis^\alpha$. The set-wise block stabilizers in $\lmlt(Q)$ is the subgroup $\lmlt(Q)_{[a]_\alpha}=\setof{h\in \lmlt(Q)}{h([a]_\alpha)=[a]_\alpha}$ (and similarly  $\dis(Q)_{[a]_\alpha}=\setof{h\in \dis(Q)}{h([a]_\alpha)=[a]_\alpha}$). Note that both $\lmlt(Q)_a$ and $\lmlt^\alpha$ are contained in $\lmlt(Q)_{[a]_\alpha}$ (and the same is true for $\dis(Q)_a$, $\dis^\alpha$ and $\dis(Q)_{[a]_\alpha}$).

The {\it Cayley kernel} of a left quasigroup $Q$ is the equivalence relation $\lambda_Q$ defined  as
$$a\,\lambda_Q\, b \quad \text{ if and only if } \quad L_a=L_b.$$ 
In general, the equivalence $\lambda_Q$ is not a congruence. If $\lambda_Q=0_Q$ then $Q$ is called {\it faithful} and if all subalgebras of $Q$ are faithful we say that $Q$ is {\it superfaithful}. In particular, if $Q/\alpha$ is faithful, then $\lambda_Q\leq \alpha$ (indeed, according to \eqref{formula for pi alpha} if $L_a=L_b$ then $L_{[a]}=L_{[b]}$). If $\lambda_Q=1_Q$, i.e. $a*b=f(b)$ for every $a,b\in Q$ where $f\in \sym{Q}$, then $Q$ is called {\it permutation left quasigroup} and denoted by $(Q,f)$. If $f$ is the identity mapping then $a*b=b$ for every $a,b\in Q$ i.e. $Q$ is a {\it projection} left quasigroup. We denote by $\mathcal{P}_n$ the projection left quasigroup of size $n$ and we call {\it trivial left quasigroup} the one-element projection left quasigroup.


A {\it quasigroup} is an algebra $(Q,\ast,\backslash,/)$ such that $(Q,\ast,\backslash)$ is a left quasigroup (the {\it left quasigroup reduct} of $Q$) and $(Q,\ast,/)$ is a right quasigroup, i.e. also the right multiplications $R_a:b\mapsto b\ast a$ are bijective for every $a\in Q$. A left quasigroup is  {\it latin} if it is the left quasigroup reduct of a quasigroup (in the finite case its multiplication table is a latin square). Note that congruences and subalgebras of a quasigroup and of its left quasigroup reduct might be different since we are considering a different signatures. Nevertheless they coincide in the finite case, since the two algebraic structures are term equivalent. We introduce this rather technical distinction in order to make clear that the results of the paper are tied to the choice of the left quasigroup signature (this detail will be more relevant in the related paper \cite{Maltsev_paper}).

Latin left quasigroups are superfaithful. Indeed if $Q$ is a latin left quasigroup and $a*x=b*x$ for some $a,b,x\in Q$ then $a=b$.

A left quasigroups $Q$ is said to be {\it idempotent} if $x*x\approx x$ holds and {\it involutory} if $x*(x*y)\approx y$ holds.

Let $(A,+)$ be an abelian group, $g\in \End{A}$, $f\in \aut{A}$ and $c\in A$. We denote by $\Aff(A,g,f,c)$ the left quasigroup $(A,\cdot)$ where $x\cdot y=g(x)+f(y)+c$ and we call such left quasigroup {\it affine} over $A$. If $\Aff(A,f,g,c)$ is idempotent, then necessarily $c=0$ and $g=1-f$, so we denote it just by $\aff(A,f)$.

\subsection{Connected left quasigroup}\label{Sec:connected}
In this section we introduce the classes of connected and superconnected left quasigroups.
\begin{definition}\label{Def: Homogeneous Connected}
A left quasigroup $Q$ is said to be: 
\begin{itemize}
\item[(i)] {\it connected} if $\lmlt(Q)$ acts transitively on Q.
\item[(ii)] {\it Superconnected} if every subalgebra of $Q$ is connected.
\end{itemize}
\end{definition}

The following is a criterion for connectedness for left-quasigrops. The proof of the same criterion for racks stated in \cite[Proposition 1.3]{GB} can be employed for left quasigroups.
\begin{lemma}\label{criterion for connectedness}
Let $Q$ be left quasigroup and $\alpha\in Con(Q)$. Then $Q$ is connected if and only if $Q/\alpha$ is connected and $\lmlt(Q)_{[a]_\alpha}$ is transitive on $[a]_\alpha$ for every $a\in Q$. 
\end{lemma}
 
The property of being superconnected is determined by the connectedness of the two-generated subalgebras.
\begin{lemma}\label{2gen superconnected}
Let $Q$ be a left quasigroup. The following are equivalent:
\begin{itemize}
\item[(i)] $Q$ is superconnected.
\item[(ii)] $Sg(a,b)$ is connected for every $a,b\in Q$.
\end{itemize}
\end{lemma}

\begin{proof}
The forward implication is clear. To prove the converse, let $M$ be a subalgebra of $Q$ and $a,b\in M$. The subgroup $\lmlt(Sg(a,b))$ is transitive on $Sg(a,b)$ and then so in particular there exists $h\in \langle L_c,\, c\in Sg(a,b)\rangle\leq\lmlt(M)$ such that $h(a)=b$. Therefore $M$ is connected.
\end{proof}

The {\it orbit decomposition} $\mathcal{O}_Q$ defined by the action of $\lmlt(Q)$ (as $a\, \mathcal{O}_Q \, b$ if and only if $a$ and $b$ are in the same orbit with respect to the action of ${\lmlt(Q)}$) is a congruence of $Q$ and $Q/\mathcal{O}_Q$ is a projection left quasigroup \cite[Lemma 1.8]{semimedial}. 

\begin{proposition}\label{pi_0}
Let $Q$ be a left quasigroup and $\alpha\in Con(Q)$. Then $Q/\alpha$ is a projection left quasigroup if and only if $\mathcal{O}_Q \leq \alpha$. In particular, $Q$ is connected if and only if $\mathcal{P}_2\notin \textbf{H}(Q)$.
\end{proposition}
\begin{proof}
 If $\mathcal{O}_Q \leq \alpha $, then $Q/\alpha\cong \left(Q/\mathcal{O}_Q\right)/\left(\alpha/\mathcal{O}_Q\right)$. Therefore, $Q/\alpha$ is a projection left quasigroup. On the other hand, if $Q/\alpha$ is a projection left quasigroup, by virtue of \eqref{formula for pi alpha}, then $[h( a)]_\alpha=\pi_{\alpha}(h)([ a]_\alpha) = [a]_\alpha$ for every $a\in Q$ and $h\in \lmlt(Q)$. Hence, $\mathcal{O}_Q\leq \alpha$. 
 
 A left quasigroup is connected if and only if $Q/\mathcal{O}_Q$ is trivial, i.e. $Q$ has no proper projection factor. 
\end{proof}
\begin{corollary}\label{cor_superconnected}
A left quasigroup $Q$ is superconnected if and only if $\mathcal{P}_2\notin \textbf{{H}}\textbf{{S}}(Q)$.
\end{corollary}

The class of connected left quasigroups is closed under $\textbf{H}$, but it is not a closed under $\textbf{S}$ (for instance it is easy to find connected left quasigroups with projection subalgebras). The class of superconnected left quasigroups is closed under $\textbf{{S}}$ and $\textbf{{H}}$. On the other hand it is not closed under $\textbf{P}$ (e.g. the permutation left quasigroup $Q=(\mathbb{Z}_m,+1)$ is superconnected, but $Q^2$ is not even connected). %

 %
 %
%
%
%
%


The property of being latin is also related to the properties of $2$-generated subalgebras (similarly to what happens for superconnectedness in Lemma \ref{2gen superconnected}). 

\begin{lemma}\label{2gen latin}
Let $Q$ be a left quasigroup. If $Sg(a,b)$ is a finite latin left quasigroup for every $a,b\in Q$ then $Q$ is latin.
\end{lemma}

\begin{proof}
Assume that $x*a=y*a$. Then $x*a=y*a\in U=Sg(a,y)\cap Sg(a,x)$, which is finite and latin and so $R_a(U)=U$. Hence, $x=y$ and right multiplications are injective. For every $a,b\in Q$ there exists $x\in Sg(a,b)$ for which $x*a=b$ and so right multiplications are surjective. 
\end{proof}

\begin{example}\label{Z and Q}
\text{}
\begin{itemize}
\item[(i)] If a quasigroup $(Q,*,\backslash,/)$ and its left quasigroup reduct $(Q,*,\backslash)$ are term equivalent then $(Q,*,\backslash)$ is superconnected. Hence, any finite latin left quasigroup is superconnected. The converse is not true, as witnessed by the following superconnected non-latin left quasigroup:
\smallskip
\begin{center}
	$Q=$
\begin{tabular}{|c c c c|}
\hline
1 &2&3&4\\
        2&1&3&4\\
        3&2&1&4\\
       4&2&3&1\\
       \hline
\end{tabular}\,.
\end{center}
\medskip

\item[(ii)] Latin left quasigroup are connected but they might not be superconnected. The left quasigroup $Q=\Aff(\mathbb{Q},-1)$ is latin. The subalgebra generated by $0,1$, i.e. $\Aff(\mathbb{Z},-1)$, is a non-connected subalgebra of $Q$ (and in particular the converse of Lemma \ref{2gen latin} does not hold). 

\end{itemize}

\end{example}

	\subsection{Idempotent left quasigroups}\label{Sec:Idempotent}
The blocks of congruences of idempotent left quasigroups are subalgebras, and in particular, the classes of $\lambda_Q$ are projection subalgebras. 
The orbits of $\lmlt(Q)$ and of $\dis(Q)$ coincide, because of the structure of $\lmlt{(Q)}$ given in Lemma \ref{disQ combinatorial}.  We extend \cite[Proposition 1.4]{Principal} to the setting of idempotent left quasigroups.

\begin{Lemma}\label{all faithful}
Let $Q$ be an idempotent left quasigroup. The following are equivalent:
\begin{itemize}
\item[(i)] $Q$ is superfaithful.
\item[(ii)] $Sg(a,b)$ is superfaithful for every $a,b\in Q$.
\item[(iii)] $\mathcal{P}_2 \notin \textbf{S}(Q)$.
\end{itemize}
In particular, if $Q$ is superconnected then $Q$ is superfaithful.
\end{Lemma}
\begin{proof}
(i) $\Rightarrow$ (iii) The subalgebra $\mathcal{P}_2$ is not faithful. 

(iii) $\Rightarrow$ (i) Let $M$ be a subalgebra of $Q$. The classes of $\lambda_M$ are projection subalgebras, therefore they are trivial.

(i) $\Leftrightarrow$ (ii) The equivalence is clear: indeed $\mathcal{P}_2\in \mathcal{S}(Q)$ if and only if $Sg(a,b)\cong \mathcal{P}_2$ for some $a,b\in Q$.
\end{proof}

Note that the class of superfaithful idempotent left quasigroup is closed under $\textbf{{S}}$ and $\textbf{{P}}$. 

\begin{example}
Superconnected and latin idempotent left quasigroups are superfaithful wether both the converse implications fails. Indeed the idempotent left quasigroup $\aff(\mathbb{Z},-1)$ in Example \ref{Z and Q}(ii) is superfaithful but not connected. 
\end{example}
%


%
A class of idempotent left quasigroups $\mathcal{K}$ is said to be closed under extensions if, whenever $Q/\alpha$ and $[a]_\alpha$ belong to $\mathcal{K}$ for every $a\in Q$ then also $Q$ belongs to $\mathcal{K}$. It is easy to see that if a class is closed under extensions then it is also closed under finite direct products.

%
%

\begin{lemma}\label{remark on super}
Let $Q$ be a left quasigroup and let $Q/\alpha$ be idempotent. If $Q/\alpha$ and $[a]_\alpha$ are (super)faithful (resp. connected)) for every $a\in Q$, then $Q$ is (super)faithful (resp. connected)).
\end{lemma}

\begin{proof}
The blocks of $\alpha$ are subalgebras of $Q$ since $Q/\alpha$ is idempotent. Let $M$ be a subalgebra of $Q$. We denote by $M/\alpha$ the image of $M$ under the canonical map $Q\longrightarrow Q/\alpha$.

Assume that $Q/\alpha$ and $[a]_\alpha$ are superfaithful for every $a\in Q$. If $L_a|_M=L_b|_M$ for some $a,b\in M$ then $L_{[a]}|_{M/\alpha}=L_{[b]}|_{M/\alpha}$ and so $[a]=[b]$ since the subalgebra $M/\alpha$ of $Q/\alpha$ is faithful. Therefore $L_a|_{M\cap [a]}=L_b|_{M\cap [a]}$ which implies $a=b$ since $[a]\cap M$ is faithful. 

Assume that $Q/\alpha$ and $[a]$ are superconnected for every $a\in Q$. The relation $\beta=\alpha\cap M\times M$ is a congruence of $M$. The group 
$$L=\langle \setof{L_b }{b\in [a]_\beta}\rangle\leq \lmlt(M)_{[a]_\beta}$$
 is transitive over $[a]_\beta$ since $[a]_\beta=[a]\cap M$ is a connected subalgebra of $[a]$. So $\lmlt(M)_{[a]_\beta}$ is transitive on $[a]_\beta$, $M/\beta$ is connected and therefore $M$ is connected by virtue of Lemma \ref{criterion for connectedness}.

For faithfulness and connectedness the same argument applied to the case $M=Q$ will do.
\end{proof}
%
%
%
%

\begin{corollary}\label{extensions of super}
The class of (super)faithful (resp. connected) idempotent left quasigroups is closed under extensions.
%
\end{corollary}
%
%
%
%

%
%
The class of idempotent latin left quasigroup is not closed under extensions. For instance the following superconnected idempotent left quasigroup has a congruence with a factor of size $3$ and blocks of size $3$ which are latin, but it is not latin itself:

	\medskip

\begin{center}$Q=$
	\begin{tabular}{ |c c c | c c c | c c c  | }
		\hline
		1 &3&2&7 & 8 & 9 & 4 & 5 & 6\\
		3 &2&1&7 & 8 & 9 & 4 & 5 & 6\\
		2 &1&3&7 & 8 & 9 & 4 & 5 & 6\\
      \hline	
		7 &8&9 &4 & 6 & 5 & 1 & 2 & 3\\
		
		7 &8&9 &6 & 5 & 1 & 1 & 2 & 3\\
		
		7 &8&9 &5 & 4 & 6 & 1 & 2 & 3\\
	      \hline
		4 &5&6 &1 & 2 & 3 & 7 & 9 & 8\\
		4 &5&6 &1 & 2 & 3 & 9 & 8 & 7\\
		4 &5&6 &1 & 2 & 3 & 8 & 7 & 9\\
	
		\hline
	\end{tabular}\,.
\end{center}
%
%

\section{Racks and quandles}\label{Sec:Racks}

A \emph{rack} is a left distributive left quasigroup, i.e. a left quasigroup satisfying the identity 
\begin{equation}\label{LD}
x\ast (y\ast z)\approx (x\ast y)\ast (x\ast z).\tag{LD}
\end{equation}
An idempotent rack is a {\it quandle}. Left-distributivity \eqref{LD} implies that for a quandle $Q$, $hL_a h^{-1}=L_{h(a)}$ for every $h\in \lmlt(Q)$ and $a\in Q$. In particular, the displacement group is simply given by 
$$\dis(Q)=\langle L_a L_{b}^{-1},\, a,b\in Q\rangle.$$

\begin{example}\label{coset quandle}
\text{}
\begin{itemize}

\item[(i)] Permutation left quasigroups are racks. 

\item[(ii)] Let $G$ be a group and $H\subseteq G$ be closed under conjugation. Then $Conj(H)=(H,*)$ where $x*y=xyx^{-1}$ is a quandle.

\item[(iii)] Let $G$ be a group, $f\in  \aut{G}$ and $H \leq Fix(f)=\setof{a\in G}{f(a)=a}$. Let $G/H$ be the set of left cosets of $H$ and the multiplication defined by
	\begin{displaymath}
	aH\ast bH=af(a^{-1}b)H.
	\end{displaymath}
	Then $\Q(G,H,f)=(G/H,\ast,\backslash) $ is a quandle, called a \emph{coset} quandle. A coset quandle $\Q(G,H,f)$ is called \emph{principal} over $G$ if $H=1$ and is such case it is denoted by $\Q(G,f)$

\item[(iv)] Idempotent affine left quasigroups are quandles.	
	
\end{itemize}
\end{example}
%
%
%
%
%

If the automorphism group of a quandle $Q$ is transitive, we say that $Q$ is {\it homogeneous}. The construction in Example \ref{coset quandle}(iii) characterizes homogeneous quandles \cite{J}. For instance, connected quandles are homogeneous and they can be represented as coset quandles over their displacement group.
\begin{proposition}\cite{J, hsv} 
Let $Q$ be a connected quandle $Q$. Then $$Q\cong \mathcal{Q}(\dis(Q),\dis(Q)_a,\widehat{L_a})$$
for every $a\in Q$.
\end{proposition}

\subsection{Superfaithful quandles}\label{Sec:superfaithful}



In \cite{involutive_quandles_russo} the class of {\it $L$-groupoids} have been defined as racks such that the equation $x*a=a$ is uniquely solvable in $x$ for every $a\in Q$. According to \cite[Proposition 1]{involutive_quandles_russo} $L$-groupoids are idempotent and so they are exactly quandles with no proper projection subquandles, i.e. $L$-groupoids coincide with superfaithful quandles.

\begin{lemma}\cite[Proposition 2.4]{Principal}\label{superfQ}
Let $Q$ be a quandle. The following are equivalent:
\begin{itemize}
\item[(i)] $Q$ is superfaithful.
\item[(ii)] $\mathcal{P}_2 \notin \textbf{S}(Q)$.
\item[(iii)] $Fix(L_a)=\{a\}$ for every $a\in Q$.
\end{itemize}
\end{lemma}

%

The coset quandle construction provides a way to construct finite homogeneous superfaithful quandles.
\begin{lemma}\label{P_2 in coset}
	%
Let $Q=\mathcal{Q}(G,H,f)$ be a quandle over a finite group $G$.	If $|H|$ and $|f|$ are coprime then $|x^f|=|xH^{L_H}|$ for every $x\in G$. 
\end{lemma}

\begin{proof}
Clearly $n=|xH^{L_H}|$ divides $|x^f|$ and $f^n(x)=xa$ for some $a\in H$. Therefore $f^{sn}(x)=xa^s$ and so $|x^f|=n|a|$. Thus $|a|$ divides both $|H|$ and $|f|$ and so $a=1$, i.e. $|x^f|=n$.
\end{proof}

\begin{corollary}
Let $Q=\mathcal{Q}(G,Fix(f),f)$ be a quandle over a finite group $G$. If $|Fix(f)|$ and $|f|$ are coprime then $Q$ is superfaithful.
\end{corollary}

\begin{proof}
Apply Lemma \ref{P_2 in coset} to the case $H=Fix(f)$. Indeed if $\{H,gH\}$ is a projection subquandle, then $H\ast gH=f(g)H=gH$. Thus $f(g)=g$, i.e. $g\in H$. Therefore $Fix(L_H)=\{H\}$ and since $Q$ is homogeneous, the left multiplications have all the same cycle structure, i.e. $Fix(L_{xH})=\{xH\}$ for every $x\in G$.
\end{proof}

The converse of Lemma \ref{P_2 in coset} is not true: there exist latin quandles of size $pq$, for $p,q$ primes, with left multiplications of order $2p$ and stabilizers of size $p$ \cite{LSS} (e.g. {\tt SmallQuandle}(15,5) from the \cite{RIG} database of GAP).

Let $G$ be a group, $\theta\in \aut{G}$ and $t\in\mathbb{N}$. We define
$$\theta_t:G^t\longrightarrow G^t,\quad (x_1,\ldots,x_t)\mapsto (\theta(x_t),x_1,x_2,\ldots, x_{t-1}).$$
It is easy to check that $\theta_t$ is an automorphism of $G^t$ and that $H_t=Fix(\theta_t)=\setof{(a,a,\ldots,a)}{a\in Fix(\theta)}\cong Fix(\theta)$. We denote by $(G,t,\theta)$ the coset quandle $\mathcal{Q}(G^t,H_t,\theta_t)$.

\begin{lemma}\label{proj sub for simple}
Let $Q=(G,t,\theta)$ be a quandle and $Fix(\theta_t)=H_t$. Then
$$Fix(L_{H_t})=\setof{(b ,ba,ba^2,\ldots, ba^{t-2},ba^{t-1})H_t}{ b\in G,\,a\in Fix(\theta),\, \theta(b)=ba^{-t}}$$
%
%
	%
\end{lemma}


\begin{proof}
 We have that $H_t\ast xH_t=\theta_t(x)H_t=xH_t$ for $x=(b_1,\ldots, b_t)$ if and only $x$ is a solution to the following system of equations
	\begin{equation}\label{system of eq}
	\left\{
	\begin{array}{l}
	b_1^{-1}b_2=a\\
	b_2^{-1}b_3=a\\
	\ldots \\
	b_{t-1}^{-1}b_{t}=a\\
	\theta(b_{t})^{-1}b_{1}=a\\  
	\end{array}\right.
	\end{equation}  
	where $a\in Fix(\theta)$. The solutions to \eqref{system of eq} are $S=\setof{(b ,ba,ba^2,\ldots, ba^{t-2},ba^{t-1})}{a\in Fix(\theta),\, \theta(b)=ba^{-t}}$. 
\end{proof}

We can use the construction $(G,t,\theta)$ to produce an infinite family of superfaithful quandles out of a single one.

\begin{proposition}\label{prop on simple}
Let $G$ be a finite group, $\theta\in \aut{G}$ and $t\in \mathbb{N}$. The following are equivalent:
\begin{itemize}
\item[(i)] $(G,1,\theta)$ is superfaithful and $t$ and $|Fix(\theta)|$ are coprime.
\item[(ii)] $(G,t,\theta)$ is superfaithful.
\end{itemize}
\end{proposition}

\begin{proof}
According to Lemma \ref{proj sub for simple}, $H_t\neq xH_t\in Fix(L_{H_t})$ if and only if $x=(b ,ba,ba^2,\ldots, ba^{t-2},ba^{t-1})$ for some $1\neq a\in Fix(\theta)$ and $\theta(b)=ba^{-t}$.

Assume that $t$ and $|Fix(\theta)|$ are coprime. Under this assumption, the mapping
$$\langle a\rangle \longrightarrow \langle a\rangle, \quad x\mapsto x^t$$
is a bijection for every $a\in Fix(\theta)$ and $x^t=1$ if and only if $x=1$. So, if $\theta(b)=ba$ holds for some $b\in G$ and $1\neq a \in Fix(\theta)$ then $a=c^t$ for some $c\neq 1$. On the other hand if $\theta(b)=ba^t$ for some $a\neq 1$ then also $a^t\neq 1$. Therefore $(G,t,\theta)$ is superfaithful if and only if $(G,1,\theta)$ is superfaithful. Thus, the implication (i) $\Rightarrow$ (ii) holds.

To complete the proof of the implication (ii) $\Rightarrow$ (i), we need to show that if $(G,t,\theta)$ is superfaithful then $t$ and $Fix(\theta)$ are coprime.  Assume that $p$ is a prime dividing $|Fix(\theta)|$ and $t$. Then there exists $a\in Fix(\theta)$ of order $p$ and $a^{t}=a^{pt^\prime}=1$. So $\theta(1)=1=1\cdot a^{t}$ and so $(1,a,a^2,\ldots,a^{t-1})H_t\in Fix(L_{H_t})\neq \{H_t\}$.
\end{proof}

\begin{example}\label{ex on simple2} 
Let $G$ be a finite group and $Q=(G,t,\theta)$.
\begin{itemize}
\item[(i)] If $Core_G(Fix(\theta))=1$ and $t$ and $|Fix(\theta)|$ are not coprime, then $Q$ is faithful but not superfaithful.

\item[(ii)] Let $\theta=1$. Then $Q=(G,t,1)$ is superfaithful if and only if $|G|$ and $t$ are coprime. In particular, if $G$ is simple then $Q$ is a simple quandle (thus simple quandles are faithfut but not necessarily superfaithful).
\end{itemize}

\end{example}

Recall that for a quandle $Q$, the equivalence relation $\sigma_Q$ is defined by
\begin{equation}\label{sigma}
a\, \sigma_Q\,b \text{ if and only if } \dis(Q)_a=\dis(Q)_b.
\end{equation}
The blocks of $\sigma_Q$ are subquandles and they are also blocks with respect to the action of $\lmlt(Q)$ \cite[Section 2.3]{Principal}.

\begin{proposition}\label{principal dec}
Let $Q$ be a finite superfaithful quandle. Then $[a]_{\sigma_Q}$ is a principal latin quandle over $N_{\dis(Q)}(\dis(Q)_a)/\dis(Q)_a$ and $N_{\dis(Q)}(\dis(Q)_a)/\dis(Q)_a$ is solvable for every $a\in Q$.
\end{proposition}

\begin{proof}
The block $S=[a]_{\sigma_Q}$ is a finite superfaithful semiregular quandle. Then $S$ is latin and in particular it is connected \cite[Corollary 2.9]{Principal}. Hence $[a]_{\sigma_Q}$ is contained in the orbit of $a$ with respect to $\dis(Q)$. Hence, according to \cite[Theorem 3.4]{Principal}, $[a]_{\sigma}$ is principal over $\dis(S)\cong N_{\dis(Q)}(\dis(Q)_a)/\dis(Q)_a$ that is solvable (the diplacement group of a finite latin quandle is solvable \cite{Stein2}).             
\end{proof}

According to Proposition \ref{principal dec}, finite superfaithful quandles are the disjoint union of principal latin quandles. Note that such a partition can be trivial.

\subsection{Superconnected racks}\label{Sec:Superconnected}

Finite latin quandles are superconnected, but the converse implication fails, although examples seem to be rare. Examples of superconnected non-latin quandles are provided by the family {\it locally strictly simple quandles} studied in \cite{LSS}. The smallest such quandles are {\tt SmallQuandle}(28,i) with $i=3,4,5,6$ in the  \cite{RIG} library of GAP (such quandles have a congruence with latin blocks and latin factor, but they are not latin).

\begin{example}\label{perm rack2}
\text{}
\begin{itemize}

\item[(i)] A permutation rack $(Q,f)$ is connected if and only if $Q=\setof{f^j(a)}{j\in \mathbb{Z}}$ for every $a\in Q$. 
 Let $C$ be a cyclic group generated by $c$ and $|C|=|Q|$. The map
$$(Q,f)\longrightarrow \Aff(C,0,1,c),\quad f^j(a)\mapsto jc$$
is an isomorphism of racks.   
The rack $\Aff(C,0,1,c)$ is generated by any of its elements and then it is superconnected (note that every monogenerated rack arise in this way). 
%
%
\item[(ii)] Let $\mathcal{C}$ be a conjugacy class in a group $G$. Then $\mathcal{C}$ is superconnected if and only if every pair of elements $a,b\in \mathcal{C}$ are conjugate in the subgroup $\langle a,b\rangle$ (see Lemma \ref{2gen superconnected}).

\end{itemize}
\end{example}

 Let $Q$ be a rack, $A$ an abelian group, $\psi\in \aut{A}$ and a map $\theta:Q\times Q\longrightarrow A$. We define the left quasigroup $E=Q\times_{\psi,\theta}A=(Q\times A,*)$ where
\begin{equation}\label{central ext op}
(a,s)*(b,t)=(a*b,(1-\psi)(s)+\psi(t)+\theta_{a,b})
\end{equation}
for every $a,b\in Q$ and $s,t\in A$. Under suitable conditions on $\theta$ and $\psi$, $E$ is a rack \cite[Section 7]{CP} and we say that $E$ is a {\it central extension} of $Q$ by $A$. The projection onto $Q$ is a rack morphism and if $\psi=1$, then its kernel is contained in the congruence $\lambda_{E}$. In this case, following \cite{Eisermann} and \cite{covering_paper}, we say that $E$ is an {\it abelian cover} of $Q$.

Recall that for a rack $Q$ the equivalence relation $\m_Q$ which blocks are $[a]_{\m_Q}=Sg(a)$ is a congruence of $Q$ contained in $\lambda_Q$ \cite[Proposition 7.1]{covering_paper}.
\begin{proposition}
Let $Q$ be a rack. The following are equivalent:
\begin{itemize}
\item[(i)] $Q$ is superconnected.
\item[(ii)] $Q/\m_Q$ is superconnected.
\end{itemize}
If particular, if $Q$ is superconnected then $\lambda_Q=\m_Q$ and $Q$ is an abelian cover of $Q/\m_Q$. 
\end{proposition}

\begin{proof}
The blocks of $\m_Q$ are subracks since $Q/\m_Q$ is idempotent. The block $[a]_{\m_Q}$ is the subrack generated by $a$ and so it superconnected according to Example \ref{perm rack2}(i). So, the equivalence between (i) and (ii) follows by Lemma \ref{remark on super}.
%

If $Q$ is superconnected, then $Q/\m_Q$ is superconnected and so $Q$ is faithful. Therefore, $\lambda_Q\leq\m_Q\leq\lambda_Q$ and so equality holds. Finally, $Q$ is superconnected and then homogeneous, so we can apply \cite[Corollary 7.1(5)]{covering_paper}, i.e. $Q$ is an abelian cover of $Q/\m_Q$. 
\end{proof}

Some of the contents of Section 2.4 of \cite{Principal} on principal latin quandles extend to principal superconnected quandles. 
%
%
%

\begin{proposition} \label{from Principal}
Let $Q=\Q(G,f)$ be a superconnected quandle. 
\begin{itemize}
\item[(i)] The subquandles of $Q$ are coset with respect to $f$-invariant subgroups of $G$ and they are principal.
\item[(ii)] $\dis_\alpha=\dis^\alpha=\dis(Q)_{[a]_\alpha}$ for every $\alpha\in Con(Q)$.
\item[(iii)]  $Con(Q)\cong \setof{N\unlhd G}{f(N)=N}$ and $Q/\alpha$ is principal for every $\alpha\in Con(Q)$.
\end{itemize}
\end{proposition}
\begin{proof}
All  subquandles of $Q$ are connected, then we can apply \cite[Lemma 2.7]{Principal} for (i), and \cite[Corollary 2.11]{Principal} for (ii) and (iii).
\end{proof}

\subsection{Commutator theory for superconnected quandles}\label{Sec:commutator}


According to the commutator theory developed in \cite{comm} we can define abelianess and centrality for congruences of arbitrary algebraic structures (e.g. for left quasigroups) using the commutator between congruences (we omit the general definition and we denote the commutator between two congruences $\alpha,\beta$ by $[\alpha,\beta    ]$). Consequently, nilpotence and solvability are defined by using a special chain of congruences defined in analogy with the derived series and the lower central series of groups. The derived series of a left quasigroup $Q$ is defined as
\begin{displaymath}
\gamma^0(Q)=1_Q,\qquad \gamma^{n+1}(Q)=[\gamma^n(Q),\gamma^n(Q)],
\end{displaymath}
and the lower central series as
\begin{displaymath}
\gamma_0(Q)=1_Q,\qquad \gamma_{n+1}(Q)=[\gamma_n(Q),1_Q],
\end{displaymath}
for $n\in\mathbb{N}$. A left quasigroup is solvable (resp. nilpotent) of length $n$ if $\gamma^n(Q)=0_Q$ (resp $\gamma_n(Q)=0_Q$). 

In \cite{CP, semimedial} we adapted the theory of Freese and McKenzie to racks. To this end, we make use of a Galois connection between the congruence lattice of a rack $Q$ and the lattice of {\it admissible subgroups} $Norm(Q)=\setof{N\trianglelefteq \lmlt(Q)}{N\leq \dis(Q)}$. The pair of mappings $\alpha\mapsto\dis_\alpha$ and $N\mapsto\c{N}=\setof{(a,b)\in Q\times Q}{L_a L_b^{-1}\in N}$ provide a monotone Galois connection between $Con(Q)$ and $Norm(Q)$. 
If the mappings $\dis$ and $\c{}$ are mutually inverses lattice isomorphism, we say that $Q$ has the CDSg property (see \cite[Section 3.4]{CP}). For racks abelianess and centrality of congruences are completely determined by the properties of the relative displacement groups \cite[Theorem 1.1]{CP}. In particular, a quandle is nilpotent (resp. solvable) if and only if its displacement group is nilpotent (resp. solvable) \cite[Theorem 1.2]{CP}.

Some of the results stated in \cite{CP} for finite latin quandles, actually apply to the class of superconnected quandles, showing that connectedness is a relevant property also from the commutator theory viewpoint.

\begin{proposition}
Principal superconnected quandles have the CDSg property.
\end{proposition}

\begin{proof}
Every superconnected quandle is faithful. If $Q$ is a principal superconnected quandle and $\alpha\in Con(Q)$, then by virtue of Proposition \ref{from Principal}(ii) $\dis^\alpha=\dis_\alpha$. Thus, we can apply \cite[Proposition 3.13]{CP} and it follows that principal superconnected quandles have the CDSg property.
\end{proof}

\begin{proposition}\label{commutator for superconnected}
Let $Q$ be a superconnected quandle. Then: 
\begin{enumerate}
\item[(i)] $\alpha=\c{\dis_\alpha}$ and $[\alpha,\beta]=[\beta,\alpha]=\c{[\dis_{\alpha},\dis_\beta]}$ for every $\alpha,\beta\in Con(Q)$.	
\item[(ii)] The mapping $\dis$ is injective and the mapping $\c{}$ is surjective.
	\end{enumerate}
\end{proposition}

\begin{proof}
	All factors of superconnected quandles are superconnected and then faithful. So, according to \cite[Proposition 3.7]{CP} we have that $\alpha\leq \c{\dis_\alpha}\leq \c{\dis^\alpha}=\alpha$ and so $\alpha=\c{\dis_\alpha}$ for every $\alpha\in Con(Q)$. For the other statements, we can apply directly \cite[Propositions 3.8 and 5.5]{CP} since all the factor of $Q$ are faithful.
\end{proof}


Nilpotent superconnected quandles are indeed latin. 
\begin{theorem}\label{nilpotence}
Nilpotent superconnected quandles are latin.
\end{theorem}

\begin{proof}
If $Q$ is abelian and superconnected, then it is faithful and connected and so latin \cite[Corollary 2.6]{Principal}. Let $Q$ be nilpotent of length $n+1$, i.e. $\gamma_n(Q)$ is central. The group 
$$D=\langle L_b L_a^{-1},\, b\in [a]_{\gamma_n(Q)}\rangle\leq \dis_{\gamma_n(Q)}$$
 is transitive on the block of $[a]_{\gamma_n(Q)}$ and $Q$ is connected. Then we can apply \cite[Proposition 7.8]{CP} and we have that $Q$ is a central extension of $Q/\gamma_n(Q)$, i.e. the quandle operation of $Q$ is defined as in \eqref{central ext op} by
 $$R_{(a,s)}(b,t)=(b,t)*(a,s)=(R_a(b),((1-\psi)(t)+\psi(s)+\theta_{a,b}).$$
  By induction on the nilpotency length, $Q/\gamma_n(Q)$ is latin and the blocks of $\gamma_n(Q)$ are abelian and therefore latin, i.e. $1-\psi$ is bijective. Therefore the right multiplication $R_{(a,s)}$ has inverse $R_{(a,s)}^{-1}(b,t)=(R_a^{-1}(b),(1-\psi)^{-1}(t-\psi(s)-\theta_{R_a^{-1}(b),a})$ and so $Q$ is latin.
\end{proof}

The converse of Theorem \ref{nilpotence} does not hold. Indeed there exist infinite affine latin quandles which are not superconnected (see Example \ref{Z and Q}(ii)). 

%
%
%

%
%
%
%
%
%

The biggest central congruence of a quandle $Q$ is called the {\it center of $Q$} (the analog of the center of a group) and denoted by $\zeta_Q$. According to \cite[Proposition 5.8]{CP} (and \cite[Proposition 3.14]{semimedial}) we have 
$$a\,\zeta_Q\, b \, \text{ if and only if } \, L_a L_b^{-1}\in Z(\dis(Q)) \, \text{ and } \,a\, \sigma_Q\,b,$$
where $\sigma_Q$ is defined in \eqref{sigma}.

\begin{lemma}\label{proj sub}
	Let $Q$ be a finite connected quandle and $\alpha\leq \zeta_Q$. If $Q$ is superfaithful then $Q/\alpha$ is superfaithful.
\end{lemma}

\begin{proof}
Assume that $[a]*[b]=[b]$. Then $L_a L_b^{-1}\in \dis(Q)_{[b]}$. According to \cite[Corollary 3.2]{GB} the block stabilizer is the direct product of $\dis_\alpha=\setof{L_c L_b^{-1}}{c\, \alpha\,b}$ and the stabilizer of $b$ in $\dis(Q)$. Thus, there exists $c\,\alpha\,b$ and $h\in \dis(Q)_b=\dis(Q)_c$ such that $L_a L_b^{-1}=h L_c L_b^{-1}$. Then $L_a=h L_c\in \lmlt(Q)_c$ and accordingly $a*c=c$. Then, $a=c\,\alpha\, b$ and so $Fix(L_{[a]})=\{[a]\}$ and Lemma \ref{superfQ} applies.
\end{proof}

\begin{proposition}\label{nilpotent latin}
	Let $Q$ be a finite nilpotent quandle. The following are equivalent:
\begin{itemize}
\item[(i)] $Q$ is connected and superfaithful.
\item[(ii)] $Q$ is latin. 
\end{itemize}	
\end{proposition}

\begin{proof}
(i) $\Rightarrow$ (ii) Let us proceed by induction on the nilpotency length. If $Q$ is abelian then it follows by \cite[Corollary 2.6]{Principal}, since $Q$ is connected and faithful. Let $Q$ be nilpotent of length $n$. By Lemma \ref{proj sub}, $Q/\zeta_Q$ is superfaithful and then by induction $Q/\zeta_Q$ is latin. So we can apply \cite[Lemma 3.4]{GB} and conclude that $Q$ is latin.

(ii) $\Rightarrow$ (i) True in general.
\end{proof}

Note that the superconnected quandles of size 28 mentioned earlier in the previous section are solvable but not latin, so Theorem \ref{nilpotence} does not extend to the solvable case.  Nevertheless finite solvable superconnected quandles have the {\it Lagrange property}, i.e. the size of every subalgebra divides the size of the quandle (extending a known result for left distributive quasigroups \cite{GALK3}).

In the proof of the following Proposition we are using that the blocks of a congruence of a connected left quasigroup have all the same size \cite[Lemma 2.5]{CP}.

\begin{Proposition}
Finite solvable superconnected quandles have the Lagrange property.\end{Proposition}

\begin{proof}
If $Q$ is abelian, the statement is true because subquandles correspond to submodules with respect to the structure given by the affine representation \cite[Proposition 2.18]{Principal}. Let $Q$ be solvable of length $n+1$ i.e. $\gamma^{n}(Q)$ is an abelian cogruence and let $M$ be a subquandle. Then $|M|=|M/\gamma^{n}(Q)||[a]\bigcap M|$. Since $[a]$ is affine, $|M\bigcap [a]|$ divides $|[a]|$ and since $Q/\gamma^{n}(Q)$ is solvable of length $n$, by induction we have that $|M/\gamma^{n}(Q)|$ divides $|Q/\gamma^{n}(Q)|$. Therefore, $|M|$ divides $|Q|=|Q/\gamma^{n}(Q)||[a]|$. 
\end{proof}

\section{Involutory quandles}\label{Sec:Involutory}

\subsection{Two generated involutory quandles}

Recall that a quandle $Q$ satisfying the identity $x*(x*y)\approx y$ is called \emph{involutory}. 

The local properties of $2$-generated subquandles determine the global properties such as superconnectedness and latinity (see Lemmas \ref{2gen superconnected} and \ref{2gen latin}). A description of the free involutory quandle on $2$ generators is given in \cite[Corollary 10.4]{J}, namely such quandle is isomorphic to $\Aff(\mathbb{Z},-1)$.

We investigate the properties of involutory quandles according to the properties of the canonical generators of the displacement group. A similar approach was take in \cite{Nobu} using the concept of {\it cycle}. The main original contribution is to partially extend the main result of \cite{Nobu} to the infinite case.

Let $Q$ be an involutory quandle and $a,b\in Q$. Following \cite{Nobu} we define the {\it cycle generated by $a$ with base $b$} as
$$C(a,b)=\setof{a^k}{k\in \mathbb{Z}},\, \text{ where }\, a^k=\begin{cases} (L_a L_b)^i(b),\, \text{ if }k=2i,\\(L_a L_b)^i(a),\, \text{ if }k=2i+1. \end{cases}.$$
According to \cite[Corollary 5.4]{semimedial}, $Sg(a,b)=a^{L_a L_b}\cup b^{L_a L_b}$ and so we have that $C(a,b)$ is the subquandle generated by $a,b$. If $|L_a L_b|$ is finite we can define $ord_{a,b}=\min_{k>0 }\{a^k=b\}$.

\begin{proposition}\label{free involutory}
Let $Q$ be an involutory quandle generated by $a,b\in Q$. Then $\dis(Q)$ is the cyclic group generated by $L_a L_b$.
\end{proposition}

\begin{proof}
 According to \cite[Corollary 10.4]{J}, the free $2$-generated involutory quandle $F$ is isomorphic to $\Aff(\mathbb{Z},-1)$ that is generated by $0$ and $1$. Since we have that 
$$L_a L_b(c)=2(a-b)+c$$ 
for every $a,b,c\in F$, the displacement group of $F$ is $\langle L_1 L_0\rangle=2\mathbb{Z}\cong \mathbb Z$. The canonical surjective quandle homomorphism $F\longrightarrow Q$, induces a surjective group homomorphism $\mathbb{Z}\longrightarrow \dis(Q)$ and so $\dis(Q)$ is cyclic and it is generated by $L_a L_b$. 
\end{proof}

\begin{lemma}\label{gen of dis}
Let $Q$ be an involutory quandle, $a,b\in Q$ and $n\in \mathbb{N}$. Then:
\begin{itemize}
\item[(i)] $(L_a L_b)^{2n+1}=L_a L_{(L_b L_a)^n(b)}$.
\item[(ii)] $(L_a L_b)^{2n}=L_a L_{(L_b L_a)^n(a)}$.
\end{itemize}
\end{lemma}

\begin{proof}
Let $a,b\in Q$. For (i) we have
$$(L_a L_b)^{2n+1}=L_a (L_b L_a)^n L_b (L_a L_b)^n=L_a (L_b L_a)^n L_b (L_b L_a)^{-n}=L_a L_{(L_b L_a)^n(b)}$$
For (ii): 
\begin{align*}
(L_a L_b)^{2n}&=(L_a L_b)^{n}L_a^2(L_a L_b)^n =L_a (L_b L_a)^n L_a (L_a L_b)^n\\
&=L_a (L_b L_a)^n L_a (L_b L_a)^{-n}=L_a L_{(L_b L_a)^n(a)}\qedhere
\end{align*}
%
%
\end{proof}

\begin{corollary}\label{finiteness}
Let $Q$ be an involutory quandle and $a,b\in Q$. The following are equivalent:

\begin{itemize}
\item[(i)] $|L_a L_b|$ is finite.
\item[(ii)] $Sg(a,b)$ is finite 
\item[(iii)] $ord_{a,b}$ is finite.
\end{itemize}

\end{corollary}

\begin{proof}
Let $S=Sg(a,b)$. According to Lemma \ref{free involutory}, $\dis(S)$ is generated by $L_a L_b$ and so, by \cite[Corollary 5.4]{semimedial}, $S=a^{L_a L_b}\cup b^{L_a L_b}$. 

(i) $\Rightarrow$ (ii) Clearly if $|L_a L_b|$ is finite then $S$ is finite. 

(ii) $\Rightarrow$ (iii) If $S$ is finite then $ord_{a,b}\leq 2|a^{L_a L_b}|\leq |S|$ is finite.

(ii) $\Leftrightarrow$ (iii) It follows by Lemma \ref{gen of dis}. Indeed, if $s=ord_{a,b}$ is even then $(L_a L_b)^{2s}=L_a L_{(L_b L_a)^s(a)}=L_a^2=1$. If $s$ is odd then $(L_a L_b)^{2s+1}=L_a L_{(L_b L_a)^2(b)}=L_a^2=1$.
\end{proof}

\begin{proposition}\label{2 gen involutory}
Let $Q$ be an involutory quandle and $a,b\in Q$. The following are equivalent:
\begin{itemize}
\item[(i)] $Sg(a,b)$ is connected.
\item[(ii)] $Sg(a,b)$ is a finite latin quandle of odd order.
\item[(iii)] $ord_{a,b}$ is finite and odd.
\end{itemize}
\end{proposition}

\begin{proof}

According to Lemma \ref{free involutory}, $S=Sg(a,b)$ has cyclic displacement group generated by $L_a L_b$ and every orbit of $S$ is isomorphic to $\aff(C,-1)$ where $C$ is a cyclic group.

(i) $\Rightarrow$ (ii) If $Q$ is connected, then $Q\cong \Aff(C,-1)$. If $C$ is infinite then $Q$ is not connected. Hence $Q$ is a finite connected affine quandle, and then $Q$ is latin (see \cite{Principal, hsv}). In particular, $R_0:x\mapsto 2x$ is bijective and so $|C|$ is odd.
%
%
%
%

(ii) $\Rightarrow$ (iii) If $Q$ is a finite latin quandle, then $Q\cong \aff(\mathbb{Z}_{2n+1},-1)$ for some $n$. The condition $(L_1 L_0)^n(0)=2n=0$ is satisfied just for $n=0$.

(iii) $\Rightarrow$ (i) If $ord_{a,b}$ is odd, then there exists $n=2k+1$ such that $b=(L_a L_b)^k(a)$ and so $Q$ is connected.
\end{proof}
%
%
%

%

\subsection{Superconnected and latin involutory quandles}

Let us first note that in one direction \cite[Proposition 6]{Nobu} works also for infinite superfaithful quandle (indeed the proof indeed just requires that the order of the canonical generators of the displacement group have finite order and that the two-generated subquandles are faithful). We provide an alternative proof using Lemma \ref{2 gen involutory}.

\begin{lemma}
\label{from nobu}
Let $Q$ be an involutory quandle and $a,b\in Q$ such that $|L_a L_b|$ is finite. If $Q$ is superfaithful then $|L_a L_b|=ord_{a,b}$ is odd.
\end{lemma}

\begin{proof}
According to \cite[Proposition 5]{Nobu} $|L_a L_b|=ord_{a,b}$ since $Q$ is faithful.

The group $C=\langle L_a L_b\rangle$ is finite and so it is $S=Sg(a,b)$ Assume that $S$ is not connected, i.e. $S=a^C\cup b^C=O_a\cup O_b$. If $|O_a|$ is even then $O_a\cong \aff(\mathbb{Z}_{2m},-1)$ has projection subquandle. Then $|O_a|$ is odd and $L_b$ acts on $O_a$. Since $L_b$ has order $2$ then $L_b|_{O_a}$ has fixed points. According to Lemma \ref{superfQ}, $\mathcal{P}_2\in \mathcal{S}(Q)$, contradiction. Hence $S$ is connected and so Proposition \ref{2 gen involutory} applies. 
\end{proof}

%

The following theorem characterizes superconnected involutory quandles.

\begin{theorem}\label{involutory superfaith}
Let $Q$ be an involutory quandle. The following are equivalent:
\begin{itemize}
\item[(i)] $Q$ is superconnected.
\item[(ii)] $Q$ is latin and $|L_a L_b|$ is finite for every $a,b\in Q$.
\item[(iii)] $ord_{a,b}=|L_a L_b|$ is finite and odd for every $a,b\in Q$.
\end{itemize}
\end{theorem}


\begin{proof}

(i) $\Rightarrow$ (ii) According to Proposition \ref{2 gen involutory}, if the subquandle $Sg(a,b)$ is connected then it is finite and latin. Then we can conclude that $Q$ is latin by  Lemma \ref{2gen latin} and that $|L_a L_b|$ is finite by Corollary \ref{finiteness}.


(ii) $\Rightarrow$ (iii) Follows by Lemma \ref{from nobu} since $Q$ is superfaithful.

(iii) $\Rightarrow$ (i) By virtue of Proposition \ref{2 gen involutory}, every pair of elements of $Q$ generates a finite connected quandle. Thus we can conclude by Lemma \ref{2gen superconnected}.
%
\end{proof}
%
%
%
%
%
\begin{example}
Let $A$ be a torsion abelian group. If $A$ is $2$-divisible (i.e. $2A=A$) and $A$ has no $2$-torsion (i.e. $A$ has no elements of order $2$) then $Q=\aff(A,-1)$ is a superconnected involutory quandle. Indeed $Q$ is latin and the order of the genetators of $\dis(Q)\cong A$ is finite. For instance take $A$ to be the Pr\"ufer group $\mathbb{Z}_{p^\infty}$ for $p>2$.
\end{example}

 In the (locally) finite case we also recover the main result of \cite{involutive_quandles_russo} by using Proposition \ref{from nobu} and Theorem \ref{involutory superfaith}. Theorem also extends the main result of \cite{Nobu} to infinite involutory quandles such that the order of the canonical generators of the displacement group is finite.

\begin{corollary}\label{nobu e russi}
Let $Q$ be an involutory quandle such that $|L_a L_b|$ is finite for every $a,b\in Q$. The following are equivalent:
\begin{itemize}
\item[(i)] $Q$ is superfaithful.
\item[(ii)] $Q$ is latin.
\item[(iii)] $ord_{a,b}=|L_a L_b|$ is odd for every $a,b\in Q$.
\end{itemize}
\end{corollary}

Note that the classical result \cite[Theorem 1.2]{Gore} is exactly Corollary \ref{nobu e russi} for quandles given by conjugacy classes of involutions in finite groups.


The quandle $Q=\aff(\mathbb{Q},-1)$ is an infinite involutory latin quandle such that $|L_a L_b|$ is infinite for every $a,b\in Q$, so Corollary \ref{nobu e russi}(i) can not be pushed any further. Finite simple superfaithful involutory quandles are isomorphic to $Q\cong  \aff(\mathbb{Z}_p,-1)$ where $p$ is a prime (the unique simple latin involutory quandles). Simple involutory non-latin quandles exist, e.g. the smallest example is {\tt SmallQuandle}(10,1) from the \cite{RIG} database of GAP.

\begin{corollary}\label{solvability}
Let $Q$ be an involutory quandle such that $|L_a L_b|$ is odd for every $a,b\in Q$. Then 
\begin{itemize}
\item[(i)] $Q/\lambda_Q$ is latin.
\item[(ii)] If $Q/\lambda_Q$ is finite then $Q$ is solvable. 
\end{itemize}
\end{corollary}

\begin{proof}
(i) Assume that $|L_a L_b|=2n+1$. According to Lemma \ref{gen of dis} we have $L_a L_{(L_b L_a)^n(b)}=(L_a L_b)^{2n+1}=1$, i.e. 
$$[a]_{\lambda_Q}=[(L_b L_a)^n(b)]_{\lambda_Q}=(L_{[b]_{\lambda_Q}} L_{[a]_{\lambda_Q}})^n([b]_{\lambda_Q}).$$
Thus $Q/\lambda_Q$ is superconnected and then latin by Theorem \ref{involutory superfaith}.

(ii) If $Q/\lambda_Q$ is finite, the group $\dis(Q/\lambda_Q)$ is solvable according to the main result of \cite{Stein2} 
and $\dis^{\lambda_Q}$ is central in $\dis(Q)$ \cite[Corollary 2.3]{covering_paper}. Therefore $\dis(Q)$ is also solvable and we can conclude that $Q$ is solvable by using \cite[Lemma 6.1]{CP}.
\end{proof}

\begin{theorem}\label{nilpotent}
Let $Q$ be a finite nilpotent involutory quandle. Then $Q$ is latin if and only if $Q$ is connected and faithful.
\end{theorem}

\begin{proof}
The forward implication is clear. Let $Q$ be faithful and connected. 
According to Corollary \ref{solvability}(i) we just need to prove that $|L_a L_b|$ is odd for every $a,b\in Q$. According to \cite[Theorem 1.4]{CP}, $Q$ decomposes as the direct product of connected quandles of prime power order. Such quandles are constructed over the $p$-Sylow subgroups of the nilpotent group $\dis(Q)$. According to \cite[Theorem 8.1]{CP} there are no connected involutory quandle of size a power of $2$. Therefore the $2$-Sylow of $\dis(Q)$ is trivial and so the order of $L_a L_b$ is odd. 
\end{proof}

Let us include the following group theoretical application in the same direction of the main result of \cite{OnCommutation}.
\begin{corollary}\label{group1}
Let $G$ be a finite group generated by a conjugacy class of involutions $\mathcal{C}$. 
\begin{itemize}
\item[(i)] If $\mathcal{C}$ contains no commuting elements then $G$ is solvable.
\item[(ii)] If $|a b Z(G)|$ is odd for every $a,b\in \mathcal{C}$ then $G$ is solvable.
\end{itemize}
\end{corollary}

%
%
%

\begin{proof}
The quandle $Q=Conj(\mathcal{C})$ is involutory and $\lmlt(Q)\cong G/Z(G)$. If $Q$ is solvable, e.g. if $Q$ is latin, then $G$ is also solvable \cite[Lemma 6.1]{CP}.

(i) If $\mathcal{C}$ has no commuting elements, then $Q$ is superfaithful and so latin. 

(ii) 
If $|L_a L_a|=|abZ(G)|$ is odd for every $a,b\in Q$, $Q$ is solvable by Corollary \ref{solvability}. 
\end{proof}

\subsection{Locally reductive involutory quandles}

In \cite{Orbits} we investigate several classes of quandles, including quandles with no connected subquandles (in some sense the dual class with respect to superconnected quandles). In the finite case, such class is defined by a set of identities. 

Let us define
 $$u_0(a,b)=a,\quad u_{n+1}(a,b)=u_n(a,b)*b.$$
 A quandle is said {\it $n$-locally reductive} if $u_n(a,b)=b$ for every $a,b\in Q$ (see \cite[Section 3.2]{Orbits}). According to the theory developed in \cite{Orbits}, for a finite quandle $Q$ the following properties are equivalent:
\begin{itemize}
\item[(i)] $Q$ is locally reductive.
\item[(ii)] $Q/\lambda_Q$ is locally reductive.
\item[(iii)] $Q$ has no (proper) connected subquandles. 
\item[(iv)] $\lmlt(Q)$ is nilpotent. 
\end{itemize} 
 In this section we offer a characterization of involutory quandles satisfying one of this conditions in terms of the properties of the canonical generators of the diplacement group.

\begin{lemma}\label{gen of dis2}
Let $Q$ be an involutory quandle, $a,b\in Q$ and $n\in \mathbb{N}$. Then $(L_a L_b)^{2^n}=L_{u_n(a,b)}L_b$.
\end{lemma}

\begin{proof}
For $n=0$, the statement is trivial. By induction
\begin{align*}
L_{u_{n+1}(a,b)}L_b &=
L_{u_{n}(a,b)*b}L_b=L_{u_n(a,b)}L_b L_{u_n(a,b)}L_b \\
&=(L_a L_b)^{2^n}(L_a L_b)^{2^n}=(L_a L_b)^{2^{n+1}}\qedhere
\end{align*}


\end{proof}

For involutory quandles, the property of being locally reductive is also determined by the properties of the canonical generators of the displacement group.

\begin{proposition}\label{reductive involutory}
Let $Q$ be an involutory quandle. Then $Q/\lambda_Q$ is $n$-locally reductive if and only if $(L_a L_b)^{2^n}=1$ every $a,b\in Q$.
\end{proposition}

\begin{proof}
The quandle $Q/\lambda_Q$ is $n$-locally reductive if and only if 
$$L_{u_b(a,b)}=[u_n(a,b)]_{\lambda_Q}=u_n([a]_{\lambda_Q},[b]_{\lambda_Q})=[b]_{\lambda_Q}=L_b$$ for every $a,b\in Q$. By Lemma \ref{gen of dis2} we have that 
$$L_{u_b(a,b)}L_b=(L_{a} L_{b})^{2^n}$$ 
Therefore $L_{u_b(a,b)}=L_b$ if and only if $(L_{a} L_{b})^{2^n}=1$. \end{proof}

\begin{corollary}\label{reductive involutory 2}
An involutory quandle $Q$ is locally reductive if and only if there exists $n\in\mathbb{N}$ such that $(L_a L_b)^{2^n}=1$ for every $a,b\in Q$.
\end{corollary}

\begin{corollary}\label{group2}
Let $G$ be a finite group generated by a conjugacy class of involutions $\mathcal{C}$. Then $G$ is nilpotent if and only if there exists $n\in \mathbb{N}$ such that $|a b Z(G)|=2^n$ for every $a,b\in \mathcal{C}$.
\end{corollary}

\begin{proof}
The quandle $Q=Conj(\mathcal{C})$ is involutory and $\lmlt(Q)=G/Z(G)$. Therefore $G$ is nilpotent if and only if $Q$ is locally reductive. Hence, we can conclude by Corollary \ref{reductive involutory 2}, using that $|L_a L_b|=|abZ(G)|$ for every $a,b\in Q$.
%
\end{proof}

\bibliographystyle{plain}
\bibliography{references}

\def\cprime{$'$} \def\cprime{$'$}
\begin{thebibliography}{10}

\bibitem{AG}
Nicol{\'a}s Andruskiewitsch and Mat{\'{\i}}as Gra{\~n}a.
\newblock From racks to pointed {H}opf algebras.
\newblock {\em Adv. Math.}, 178(2):177--243, 2003.

\bibitem{OnCommutation}
Andreas B{\"a}chle and Benjamin Sambale.
\newblock Groups whose elements are not conjugate to their powers.
\newblock {\em Archiv der Mathematik}, 110(5):447--454, 2018.

\bibitem{UA}
Clifford Bergman.
\newblock {\em Universal algebra}, volume 301 of {\em Pure and Applied
  Mathematics (Boca Raton)}.
\newblock CRC Press, Boca Raton, FL, 2012.
\newblock Fundamentals and selected topics.

\bibitem{GB}
Giuliano Bianco and Marco Bonatto.
\newblock On connected quandles of prime power order.
\newblock {\em Beitr{\"a}ge zur Algebra und Geometrie / Contributions to
  Algebra and Geometry}, 2020.

\bibitem{Orbits}
M.~Bonatto, A.~Crans, T.~Nasybullov, and G.~Whitney.
\newblock Quandles with orbit series conditions.
\newblock {\em J. Algebra}, 567:284--309, 2021.

\bibitem{LSS}
Marco {Bonatto}.
\newblock {Connected quandles of size $pq$ and $4p$}.
\newblock {\em arXiv e-prints}, page arXiv:1907.07716, July 2019.

\bibitem{Maltsev_paper}
Marco Bonatto.
\newblock Maltsev classes of left-quasigroups and quandles.
\newblock {\em arXiv e-prints}, page arXiv:1904.13388, Apr 2019.

\bibitem{Principal}
Marco Bonatto.
\newblock Principal and doubly homogeneous quandles.
\newblock {\em Monatshefte f{\"u}r Mathematik}, 191(4):691--717, 2020.

\bibitem{semimedial}
Marco Bonatto.
\newblock Medial and semimedial left quasigroups.
\newblock {\em Journal of Algebra}, 2021.

\bibitem{Rumples}
Marco {Bonatto}, Michael {Kinyon}, David {Stanovsk{\'y}}, and Petr
  {Vojt{\v{e}}chovsk{\'y}}.
\newblock {Involutive latin solutions of the Yang-Baxter equation}.
\newblock {\em Journal of Algebra}, 565:128--159, 2021.

\bibitem{covering_paper}
Marco {Bonatto} and David {Stanovsk{\'y}}.
\newblock {A Universal algebraic approach to rack coverings}.
\newblock {\em arXiv e-prints}, page arXiv:1910.09317, October 2019.

\bibitem{CP}
Marco {Bonatto} and David {Stanovsk{\'y}}.
\newblock {Commutator theory for racks and quandles}.
\newblock {\em J. Math. Soc. Japan}, 73:41--75, 2021.

\bibitem{Eisermann}
Michael Eisermann.
\newblock Quandle coverings and their {G}alois correspondence.
\newblock {\em Fund. Math.}, 225(1):103--168, 2014.

\bibitem{involutive_quandles_russo}
L.~N. Erofeeva.
\newblock On a class of groupoids.
\newblock {\em Zap. Nauchn. Sem. S.-Peterburg. Otdel. Mat. Inst. Steklov.
  (POMI)}, 305(Vopr. Teor. Predst. Algebr. i Grupp. 10):136--143, 240, 2003.

\bibitem{comm}
Ralph Freese and Ralph McKenzie.
\newblock {\em Commutator theory for congruence modular varieties}, volume 125
  of {\em London Mathematical Society Lecture Note Series}.
\newblock Cambridge University Press, Cambridge, 1987.

\bibitem{GALK3}
V.~M. Galkin.
\newblock Phi-groups and left distributive quasigroups.
\newblock {\em Preprint VINITI No. 4406-81}, 1981.

\bibitem{Gore}
D.~Gorenstein.
\newblock {\em Finite Groups}.
\newblock AMS/Chelsea Publication Series. Chelsea Publishing Company, 1980.

\bibitem{hsv}
Alexander Hulpke, David Stanovsk\'{y}, and Petr Vojt\v{e}chovsk\'{y}.
\newblock Connected quandles and transitive groups.
\newblock {\em J. Pure Appl. Algebra}, 220(2):735--758, 2016.

\bibitem{J}
David Joyce.
\newblock A classifying invariant of knots, the knot quandle.
\newblock {\em J. Pure Appl. Algebra}, 23(1):37--65, 1982.

\bibitem{Loos}
O.~Loos.
\newblock {\em Symmetric Spaces}.
\newblock Number v. 1 in Mathematics lecture note series. W. A. Benjamin, 1969.

\bibitem{RIG}
Mat{\'{\i}}as Gra{\~n}a and Leandro Vendramin.
\newblock {\em {Rig, a GAP package for racks, quandles and Nichols algebras}}.

\bibitem{Matveev}
S.~V. Matveev.
\newblock Distributive groupoids in knot theory.
\newblock {\em Mat. Sb. (N.S.)}, 119(161)(1):78--88, 160, 1982.

\bibitem{Prover9}
W.~McCune.
\newblock Prover9 and mace4.
\newblock \verb|http://www.cs.unm.edu/~mccune/prover9/|, 2005--2010.

\bibitem{Nobu}
Nobuo Nobusawa.
\newblock On symmetric structure of a finite set.
\newblock {\em Osaka J. Math.}, 11:569--575, 1974.

\bibitem{Rump}
Wolfgang Rump.
\newblock Braces, radical rings, and the quantum {Y}ang-{B}axter equation.
\newblock {\em J. Algebra}, 307(1):153--170, 2007.

\bibitem{Stanos}
David Stanovsk\'{y}.
\newblock A guide to self-distributive quasigroups, or {L}atin quandles.
\newblock {\em Quasigroups Related Systems}, 23(1):91--128, 2015.

\bibitem{Stein2}
Alexander Stein.
\newblock A conjugacy class as a transversal in a finite group.
\newblock {\em J. Algebra}, 239(1):365--390, 2001.

\bibitem{Petr2}
Izabella Stuhl and Petr Vojt\v{e}chovsk\'{y}.
\newblock Enumeration of involutory latin quandles, bruck loops and commutative
  automorphic loops of odd prime power order.
\newblock {\em Nonassociative Mathematics and its Applications, proceedings of
  the 4th Mile High Conference on Nonassociative Mathematics}.

\end{thebibliography}

\end{document}